%% file: gauge.tex
\documentclass[a4paper]{article}
\usepackage[english]{babel}
\usepackage{amsmath}
\usepackage{amssymb}
\usepackage{wasysym}

\usepackage{tikz}
\usepackage{pgfplots}

\usepackage{graphicx}
\usepackage{todonotes}

\usepackage{mathtools}
\usepackage{amsthm}
\usepackage{booktabs}
\usepackage{xspace}
\usepackage[linesnumbered,ruled,vlined]{algorithm2e}
\usepackage[bookmarks,implicit=true,colorlinks=true,citecolor=blue,linkcolor=blue,urlcolor=black,pdfborder={0 0 0}]{hyperref}
\usepackage{cleveref}
\usepackage{zibtitlepage}

\definecolor{darkblue}{rgb}{0,0,.5}
\DeclareMathOperator{\epi}{epi} 
\newcommand{\argmin}{\text{arg min}} 
\newcommand{\RR}{\mathbb R } 
\newcommand{\ZZ}{\mathbb Z } 
\newcommand{\st}{\,:\,}

\newcommand{\MICP}{MICP\xspace}

\newcommand{\xlp}{\bar{x}}
\newcommand{\xint}{x_0}
\newcommand{\xbdr}{\hat{x}}

\theoremstyle{plain}
\newtheorem{thm}{Theorem}
\newtheorem{prop}[thm]{Proposition}
\newtheorem{lemma}[thm]{Lemma}
\newtheorem{corollary}[thm]{Corollary}
\theoremstyle{definition}
\newtheorem{definition}[thm]{Definition}
\newtheorem{remark}[thm]{Remark}
\newtheorem{XxmpX}[]{Example}
\newenvironment{example}{\pushQED{\qed}\begin{XxmpX}}{\popQED\end{XxmpX}}

\newcommand{\Title}{On the Relation between the Extended Supporting Hyperplane Algorithm and Kelley's Cutting Plane Algorithm}

\newcommand{\myorcidlink}[1]{\,\href{https://orcid.org/#1}{\raisebox{-0.45ex}{\includegraphics[width=1.8ex]{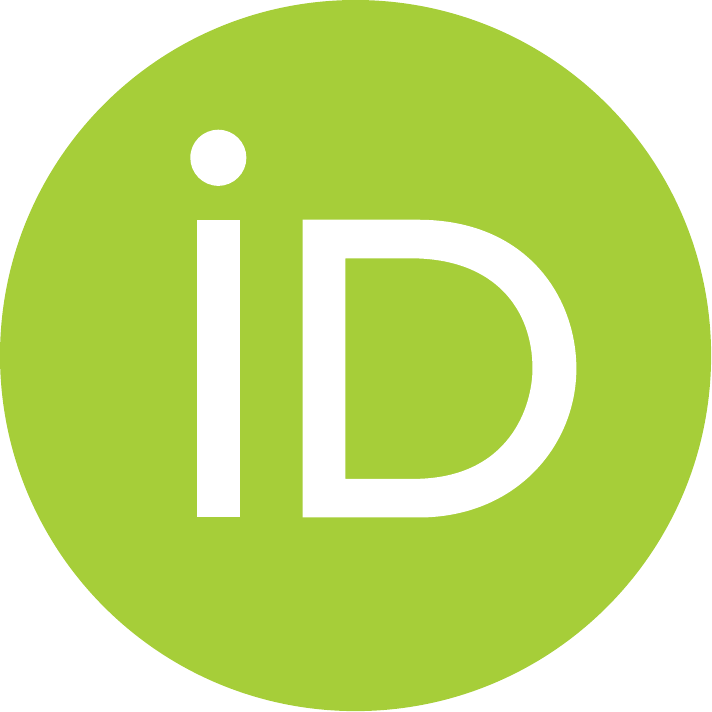}}}}

\begin{document}

\let\pdfoutorg\pdfoutput
\let\pdfoutput\undefined
\ZTPAuthor{%
  Felipe~Serrano\protect\myorcidlink{0000-0002-7892-3951},
  Robert~Schwarz\protect\myorcidlink{0000-0002-2536-026X},
  Ambros~Gleixner\protect\myorcidlink{0000-0003-0391-5903}}
\ZTPTitle{\Title}
\ZTPNumber{19-18}
\ZTPMonth{May}
\ZTPYear{2019}
\ZTPInfo{%
The work for this article has been (partly) conducted within the Research Campus
MODAL funded by the German Federal Ministry of Education and Research (BMBF
grant number 05M14ZAM).\\
The authors thank the Schloss Dagstuhl --- Leibniz Center for Informatics for
hosting the Seminar 18081 "Designing and Implementing Algorithms for
Mixed-Integer Nonlinear Optimization" for providing the environment to develop
the ideas in this paper.\\
The described research activities are funded by the Federal Ministry for
Economic Affairs and Energy within the project EnBA-M (ID: 03ET1549D).
}

\title{\Title}
\author{
  Felipe~Serrano\protect\myorcidlink{0000-0002-7892-3951}\thanks{Zuse Institute Berlin, Takustr.~7, 14195~Berlin,
Germany, \texttt{serrano@zib.de}}\and\
  Robert~Schwarz\protect\myorcidlink{0000-0002-2536-026X}\thanks{Zuse Institute Berlin, Takustr.~7, 14195~Berlin,
Germany, \texttt{schwarz@zib.de}}\and\
  Ambros~Gleixner\protect\myorcidlink{0000-0003-0391-5903}\thanks{Zuse Institute Berlin, Takustr.~7, 14195~Berlin,
Germany, \texttt{gleixner@zib.de}}}
\hypersetup{pdftitle={\Title},
  pdfauthor={Felipe Serrano, Robert Schwarz, Ambros Gleixner}}

\zibtitlepage

\maketitle


\begin{abstract}
Recently, Kronqvist et al.~\cite{KronqvistLundellWesterlund2016} rediscovered
the supporting hyperplane algorithm of Veinott~\cite{Veinott1967} and
demonstrated its computational benefits for solving convex mixed-integer
nonlinear programs.
In this paper we derive the algorithm from a geometric point of view.
This enables us to show that the supporting hyperplane algorithm is equivalent to Kelley's cutting plane
algorithm~\cite{J.E.Kelley1960} applied to a particular reformulation of the
problem.
As a result, we extend the applicability of the supporting hyperplane algorithm
to convex problems represented by general, not necessarily convex,
differentiable functions that satisfy a mild condition.
\end{abstract}

\section{Introduction}

A mixed-integer convex program (\MICP) is a problem of the form
\begin{equation}
    \label{eq:micp_gen}
    \min \{ c^T x \st x \in C \cap (\ZZ^p \times \RR^{n-p}) \}
\end{equation}
where $C$ is a closed convex set, $c \in \RR^n$, and $p$ denotes the number of
variables with integrality requirement.
The use of a linear objective function is without loss of generality given that
one can always transform a problem with a convex objective function into
a problem of the form \eqref{eq:micp_gen}.
We can represent the set $C$ in different ways, one of the most common being
as the intersection of sublevel sets of convex differentiable functions, that
is,
\begin{equation}
  \label{eq:std_rep}
  C = \{ x \in \mathbb{R}^n \st g_j(x) \leq 0, j \in J \}.
\end{equation}
Here, $J$ is a finite index set and each $g_j$ is convex and differentiable.

Several methods have been proposed for solving \MICP.
When the problem is continuous and represented as \eqref{eq:std_rep}, one of the
first proposed methods was Kelley's cutting plane
algorithm~\cite{J.E.Kelley1960}.
This algorithm exploits the convexity of a constraint function $g$ in the
following way.
The convexity and differentiability of $g$ imply that
\mbox{$g(y) + \nabla g(y) (x - y) \leq g(x)$} for every $x,y \in \RR^n$.
Since every feasible point $x$ must satisfy $g(x) \leq 0$, it follows that
\mbox{$g(y) + \nabla g(y) (x - y) \leq 0$},
for a fixed $y$, is a valid linear inequality.
If $\xlp \in \RR^n$ does not satisfy the constraint $g(x) \leq 0$, that is, if
$g(\xlp) > 0$, then
\begin{equation}
    \label{eq:gradcut}
    g(\xlp) + \nabla g(\xlp) (x - \xlp) \leq 0
\end{equation}
separates $\xlp$ from the feasible solution.
In the non-differentiable case,
\begin{equation}
    \label{eq:gradcut_nondiff}
    g(\xlp) + v^T(x - \xlp) \leq 0, \text{ with } v \in \partial g(\xlp),
\end{equation}
is also a separating valid inequality.
We will call both inequalities \eqref{eq:gradcut} and \eqref{eq:gradcut_nondiff}
\emph{gradient cut} of $g$ at $\xlp$.

The idea of Kelley's cutting plane algorithm is to approximate the feasible
region with a polytope, solve the resulting linear program (LP) and, if the LP
solution is not feasible, separate it using gradient cuts to obtain a new
polytope which is a better approximation of the feasible region and repeat, see
Algorithm \ref{alg:kelley}.
{\small
  \begin{algorithm}[h]
    \DontPrintSemicolon
    $LP = \{ x \st x \in [l , u] \}, \xlp \gets \argmin_{x \in LP} c^T x$\\
    \While{ $\max_{j \in J} g_j(\xlp) > \epsilon$ }
    {
      \ForAll{ $j$ such that $g_j(\xlp) > 0$ }
      {
        $LP \gets LP \cap \{ x \st g_j(\xlp) + \nabla g_j(\xlp) (x
        - \xlp) \leq 0 \}$
      }
      $\xlp \gets \argmin_{x \in LP} c^T x$
    }
    \Return $\xlp$
    \caption{{Kelley's cutting plane algorithm}}\label{alg:kelley}
  \end{algorithm}
}%
\\
Kelley shows that the algorithm converges to the optimum and it converges in
finite time to a point close to the optimum.
By solving integer programs (IP) using Gomory's cutting plane~\cite{Gomory1958}
instead of LP relaxations, Kelley shows that his cutting plane algorithm solves
purely integer convex programs in finite time.
The same algorithm works just as well for \MICP.
However, Kelley did not have access to a finite algorithm for solving mixed
integer linear programs (MILP).

In an attempt to speed up Kelley's algorithm, Veinott~\cite{Veinott1967}
proposes the \emph{supporting hyperplane algorithm} (SH).
A possible issue with Kelley's algorithm is that, in general, gradient cuts do
not support the feasible region, see~Figure~\ref{fig:gradient_cut}. 
Therefore, it is expected that better relaxations can be achieved by using
supporting cutting planes.

In order to construct supporting hyperplanes, Veinott suggests to build gradient
cuts at boundary points of $C$.
He uses an interior point of $C$ to find the point on the boundary, $\xbdr$,
that intersects the segment joining the interior point and the solution of the
current relaxation.
Of course, these cuts are automatically supporting hyperplanes of $C$.
However, since the cut is computed at $\xbdr$ which is in $C$, it might
happen that the gradient of the constraints active at $\xbdr$ vanishes.
For this reason, Veinott requires as a further hypothesis that the functions
representing $C$ have non-vanishing gradients at the boundary.
This is immediately implied by, e.g., Slater's condition.
Veinott also identifies that one can use his algorithm to solve
\eqref{eq:micp_gen} when representing $C$ by quasi-convex functions, that is,
functions whose sublevel sets are convex.

Recently, Kronqvist et al.~\cite{KronqvistLundellWesterlund2016} rediscovered
and implemented Veinott's algorithm~\cite{Veinott1967}.
They call their algorithm the \emph{extended supporting hyperplane algorithm} (ESH).
They discuss the practical importance of choosing a good interior point and
propose some improvements over the original method, such as solving LP
relaxations during the first iterations instead of the more expensive MILP
relaxation.
As a result, they present a computationally competitive solver implementation
for MICPs defined by convex differentiable constraint functions.

In this paper, we would like to
understand when, given a convex differentiable function $g$, gradient cuts of
$g$ are supporting to the convex set $S = \{ x \in \RR^n \st g(x) \leq 0 \}$.
This question is motivated by the fact that in this case Kelley's algorithm automatically becomes a supporting
hyperplane algorithm.
In \Cref{thm:supp} we give a necessary and sufficient condition for a gradient
cut of $g$ at a given point to be a supporting hyperplane of $S$.
In particular, this condition suggests to look at \emph{sublinear functions}, i.e.,
convex and positively homogeneous functions. 
As it turns out, this naturally leads to Veinott's algorithm.

Sublinear functions and convex sets are deeply related.
When the origin is in the interior of a convex set $S$, then we can represent $S$ via its
\emph{gauge} function~$\varphi_S$, which is sublinear~\cite{Rockafellar1970}.
We give the formal definition of the gauge function in Section~\ref{sec:gauge},
but for now it suffices to know that we can represent $S$ as $S = \{x \in \RR^n
\st \varphi_S(x) \leq 1\}$ and that, in particular, for every $\xlp \neq 0$ a
gradient cut of $\varphi_S$ at $\xlp$ supports all of its sublevel sets.
The following example illustrates this.

\begin{example} \label{ex:motivation}
  Consider the convex feasible region given by
  \begin{align*}
    S = \{ (x,y) \in \mathbb{R}^2 \st g(x,y) \leq 0\},
  \end{align*} 
  where $g(x,y) = x^2 + y^2 - 1$.
  We show through an example that gradient cuts of $g$ are not necessarily
  supporting to $S$, explain why this happens, and show that changing the
  representation of $S$ to use its gauge function solves the issue.

  Separating the infeasible point $\xlp = (\tfrac{3}{2},\tfrac{3}{2})$ by
  a gradient cut of $g$ at $\xlp$ gives
  \begin{align*}
    g(\xlp) + \nabla g(\xlp) (x - \xlp) &\le 0 \\
    \Leftrightarrow x + y &\le \dfrac{11}{6}.
  \end{align*}
  This cut does not support the circle $S$, see Figure~\ref{fig:gradient_cut}.
  Alternatively, the gauge function of the circle $S$ is given by $\varphi_S(x,y) = \sqrt{x^2
  + y^2}$ and $S = \{ (x,y) \st \sqrt{x^2 + y^2} \leq 1 \}$.
  The gradient cut of $\varphi_S$ at $\xlp$ is $x + y \le \sqrt{2}$, which is
  supporting.
\end{example}

From the previous discussion it is a natural idea to represent $C$ via its gauge
function, namely, $C = \{x \in \RR^n \st \varphi_C(x) \leq 1 \}$.
However, as mentioned before, $C$ is usually given by \eqref{eq:std_rep}.
Our main contribution is to show that reformulating
\eqref{eq:std_rep} to the gauge representation will
naturally lead to the ESH algorithm, see \Cref{subsec:evaluating}.
As a consequence, the convergence proofs of Veinott~\cite{Veinott1967} and
Kronqvist et al.~\cite{KronqvistLundellWesterlund2016} follow directly from the
convergence proof of Kelley's cutting plane algorithm~\cite{J.E.Kelley1960,
HorstTuy1990}, see~\Cref{sec:gauge_convergence}.
In other words, we show that the ESH algorithm is Kelley's cutting plane
algorithm applied to a different representation of the problem.

\begin{figure}[t]
  \centering\input{2dim_problem}
  \caption{
    The feasible region $S$ and the infeasible point $\xlp$ to separate.
    On the left we see that the separating hyperplane is not supporting to
    $S$.
    On the right we see why this happens: the linearization of $g$ at $\xlp$
    is tangent to the epigraph of $g$ (shown upside-down for clarity) at
    $(\xlp, g(\xlp))$.
    However, when this hyperplane intersects the $x$-$y$-plane, it is already
    far away from the epigraph, and in consequence, from the sublevel set.
    The intersection of the hyperplane with the $x$-$y$-plane is the gradient
    cut.
  }
  \label{fig:gradient_cut}
\end{figure}
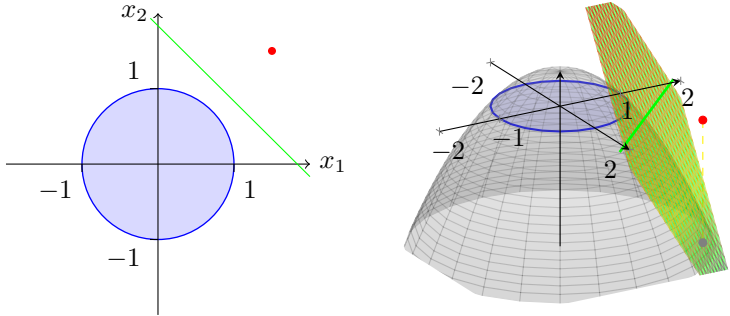

Motivated by this approach of representing
$C$ by its gauge function, we are able to show that the ESH algorithm applied to
\eqref{eq:micp_gen} converges even when $C$ is not represented by convex functions.
This is related to recent work of Lasserre~\cite{Lasserre2009} that tries to
understand how different techniques behave when the convex set $C$ is not
represented via \eqref{eq:std_rep}.
Lasserre considers sets $C = \{ x \st g_j(x) \le 0, j \in
J \}$ where $g_j$ are only differentiable, but not necessarily convex.
Under the assumption
\begin{align} \label{eq:assumption}
  \forall x \in C, \forall j \in J, g_j(x) = 0 \implies \nabla g_j(x) \neq 0,
\end{align}
that is, if the gradients of active constraints do not vanish at the boundary of
$S$, Lasserre shows that the KKT conditions are not only necessary
but also sufficient for global optimality.
In other words, every minimizer is a KKT point and every KKT point is
a minimizer.
Later, Lasserre~\cite{Lasserre2011} proposes an algorithm to find the KKT point
via log-barrier functions.
He shows that the algorithm converges to the KKT
point if \eqref{eq:assumption} holds.

Dutta and Lalitha~\cite{Dutta2011} generalized the previous result to the case
when $C$ is represented by locally Lipschitz functions, not necessarily
differentiable nor convex.\\
We show that the ESH also converges to the global optimum in the setting when
$C$ is described by differentiable functions, under \eqref{eq:assumption}.
This result extends the applicability of the SH algorithm of Veinott.

Finally, we provide a characterization of convex functions  whose linearizations
are supporting to their sublevel sets.
Although elementary, the authors are not aware of its presence in the
literature.
In particular, this result allows us to identify some families of functions for which
gradient cuts are never supporting (see~\Cref{ex:quadratic}) and some for
which they are always supporting (see Examples~\ref{ex:homo}
and~\ref{ex:linear_part}).

\paragraph{Overview of the paper.}
In the remainder of this section we introduce the notation that will be used
throughout the paper.
\Cref{sec:literature} provides a literature review on cutting plane approaches and
efforts on obtaining supporting valid inequalities.
In \Cref{sec:characterization}, we characterize functions whose
linearizations are supporting hyperplanes to their 0-sublevel sets.
\Cref{sec:gauge} introduces the gauge function and shows how to use evaluation
of the gauge function for building supporting hyperplanes.
We note that evaluating the gauge function is equivalent to the line search step
of the ESH algorithm~\cite{Veinott1967, KronqvistLundellWesterlund2016}.
This equivalence provides the link between the ESH and Kelley's cutting plane
algorithm
In \Cref{sec:gauge_convergence}, we show that the cutting planes
generated by the ESH algorithm can also be generated by Kelley's algorithm when
applied to a reformulation of the problem.
This implies that the convergence of the ESH algorithm follows from Kelley's.
In \Cref{sec:nonconvex}, we show that we can apply the ESH algorithm to problem
\eqref{eq:micp_gen} when the convex set $C$ is represented via arbitrary
differentiable functions as long as their gradients do not vanish at the
boundary of $C$.
Finally, \Cref{sec:conclusions} presents our concluding remarks.

\paragraph{Notation and definitions.}
The boundary and the interior of a set $S$ are denoted by $\partial S$ and
$\mathring S$, respectively.
The epigraph of a function $g$ is denoted by $\epi g$.
The subdifferential of a convex function $g$ at $\xlp$ is denoted by $\partial
g(\xlp)$.
Recall that the subdifferential is the set of all subgradients of $g$ at $\xlp$,
\[
  \partial g(\xlp) = \{ v \in \RR^n \st g(\xlp) + v^T (x - \xlp) \le g(x),
  \forall x \in \RR^n \}.
\]
We say that an inequality $\alpha^T x \le \beta$ is \emph{valid} for a set $S$
if every $x \in S$ satisfies $\alpha^T x \le \beta$.
Furthermore, we say that it is a \emph{supporting hyperplane} of $S$, or that it
\emph{supports} $S$, if there is an $x \in \partial S$ such that $\alpha^T
x = \beta$.

\section{Literature review}
\label{sec:literature}

We can think of the algorithms of Kelley~\cite{J.E.Kelley1960} and
Veinott~\cite{Veinott1967} as a mixture of two ingredients: which relaxation to
solve and where to compute the cutting plane.
Indeed, at each iteration, we have a point $x^k$ we would like to separate with
a linear inequality $\beta + \alpha^T(x - x_0) \leq 0$.
For Kelley's algorithm, $x_0 = x^k$, while for Veinott's algorithm, $x_0 \in
\partial C$ and for both $\alpha \in \partial g(x_0)$ and $\beta = g(x_0)$.
Choosing different relaxations and different points where to compute the cutting
planes yields different algorithms.
This framework is developed in Horst and Tuy~\cite{HorstTuy1990}.

Following the previous framework, Duran and Grossmann~\cite{DuranGrossmann1986}
propose the, so-called, outer-approximation algorithm for \MICP.
The idea is to solve an MILP relaxation but instead of computing a cutting
plane at the MILP optimum, or at the boundary point on the segment between the
MILP optimum and some interior point, they suggest to compute cutting planes at
a solution of the nonlinear program (NLP) obtained after fixing the integer
variables to the integer values given by the MILP optimal solution.
This is a much more expensive algorithm but has the advantage of finite
convergence.
Of course, this does not work in complete generality and we need some
assumptions, for example, requiring some constraint qualifications.
Moreover, we must tak care when obtaining an infeasible NLP after fixing the
integer variables in order to prevent the same integer assignment in future
iterations.
To handle such case, Duran and Grossmann propose the use of integer cuts.
However, Fletcher and Leyffer~\cite{FletcherLeyffer1994} point out that this is
not necessary and that we can use the solution of a slack NLP to build
a ``continuous'' cut that separates the integer assignment.

Westerlund and Pettersson~\cite{WesterlundPettersson1995} proposed the so-called
extended cutting plane algorithm.
This algorithm is the extension of Kelley's cutting plane
to \MICP and they show that the algorithm convergences.
Further extensions and convergence proofs of cutting plane and outer approximation
algorithms for non-smooth problems are given
in~\cite{EronenMaekelaeWesterlund2012}.

Yet another technique for producing tight cuts is to build gradient cuts at the
projection of the point to be separated onto $C$~\cite{HorstTuy1990}.
In the same reference, Horst and Tuy show that this algorithm converges.

Finally, there have been attempts at building tighter relaxations by ensuring that
gradient cuts are supporting, in a more general context than convex
mixed-integer nonlinear programming.
Belotti et al.~\cite{BelottiLeeLibertiMargotWaechter2009} consider bivariate
convex constraints of the form $f(x) - y \le 0$, where $f$ is a univariate
convex function.
They propose projecting the point to be separated onto the curve $y = f(x)$ and
building a gradient cut at the projection.
However, their motivation is not to find supporting hyperplanes, but to find the
most violated cut.
Indeed, as we will see, gradient cuts for these type of constraints are always
supporting (\Cref{ex:linear_part}).
Other work along this lines includes~\cite{LubinBienstockVielma2015}, where the authors derive an efficient procedure to
project onto a two dimensional constraint derived from a Gaussian linear chance
constraint, thus building supporting valid inequalities.

\section{Characterization of functions with supporting linearizations}
\label{sec:characterization}
We now give necessary and sufficient conditions for the linearization of
a convex, not necessarily differentiable, function $g$ at a point $\xlp$ to
support the region $S = \{ x \in \RR^n \st g(x) \le 0 \}$.
In order for this to happen, the supporting hyperplane has to support the
epigraph on the whole segment joining the point of $S$ where it supports and
$(\xlp, g(\xlp))$.
In other words, the function must be affine on the segment.
This is due to the convexity of $g$.

\begin{thm} \label{thm:supp}
  Let $g \colon \RR^n \to \RR$ be a convex function, $S = \{ x \in \RR^n \st
  g(x) \le 0 \} \neq \emptyset$, and $\xlp \notin S$.
  There exists a subgradient $v \in \partial g(\xlp)$ such that the
  valid inequality
  \begin{equation} \label{eq:lincut}
    g(\xlp) + v^T (x - \xlp) \le 0
  \end{equation}
  supports $S$, if and only if, there exists $x_0 \in S$ such that $\lambda
  \mapsto g(x_0 + \lambda(\xlp - x_0))$ is affine in $[0,1]$.
\end{thm}

\begin{proof}
  ($\Rightarrow$)
  Let $x_0 \in \partial S$ be the point where \eqref{eq:lincut} supports $S$.
  The idea is to show that the affine function $x \mapsto g(\xlp) + v^T (x
  - \xlp)$ coincides $g$ at two points, $\xlp$ and $x_0$.
  Then, by the convexity of $g$, it should coincide with $g$ on the segment
  joining both points.

  In more detail, by definition of $x_0$ we have,
  \begin{equation} \label{eq:localgradcut}
    g(\xlp) + v^T (x_0 - \xlp) = 0.
  \end{equation}
  For $\lambda \in [0,1]$, let $l(\lambda) = x_0 + \lambda (\xlp - x_0)$ and
  $\rho(\lambda) = g(l(\lambda))$.
  Since $g$ is convex and $l$ affine, $\rho$ is convex.

  Since $v$ is a subgradient,
  \[
    g(\xlp) + v^T (l(\lambda) - \xlp) \le \rho(\lambda) \text{ for every
    } \lambda \in [0,1].
  \]
  After some algebraic manipulation and using the fact that $\rho(1) = g(\xlp)
  = v^T (\xlp - x_0)$, we obtain
  \[
    \rho(1) \lambda \le \rho(\lambda).
  \]
  On the other hand, $\rho(0) = 0$ and $\rho(\lambda)$ is convex, thus we have
  $\rho(\lambda) \le \lambda \rho(1) + (1 - \lambda) \rho(0) =  \lambda \rho(1)$
  for $\lambda \in [0,1]$.
  Therefore, $\rho(\lambda) = \rho(1) \lambda$, hence $g(l(\lambda))$ is affine
  in $[0,1]$.

  ($\Leftarrow$)
  The idea is to show that there is a supporting hyperplane $H$ of $\epi
  g \subseteq \RR^{n} \times \RR$ which contains the graph of $g$ restricted to
  the segment joining $x_0$ and $\xlp$, that is, $A = \{ (x_0 + \lambda(\xlp
  - x_0), g(x_0 + \lambda(\xlp - x_0))) \st \lambda \in [0,1] \}$.
  Then, the intersection of such $H$ with $\RR^n \times \{0\}$ will give us
  \eqref{eq:lincut}.

  The set $A$ is a convex nonempty subset of $\epi g$ that does not intersect
  the relative interior of $\epi g$.
  Hence, there exists a supporting hyperplane,
  \[
    H = \{ (x,z) \in \RR^n \times \RR \st v^T x + a z = b \},
  \]
  to $\epi g$ containing $A$ (\cite[Theorem 11.6]{Rockafellar1970}).

  Since $g(x_0) \leq 0$ and $g(\xlp) > 0$, it follows that $A$ is not parallel
  to the $x$-space.
  Therefore, $H$ is also not parallel to the $x$-space and so $v \neq 0$.
  Since $A$ is not parallel to the $z$-axis, it follows that $a \neq 0$.
  We assume, without loss of generality, that $a = -1$.

  The point $(\xlp, g(\xlp))$ belongs to $A \subseteq H$, thus $v^T \xlp
  - g(\xlp) = b$ and $H = \{ (x, g(\xlp) + v^T (x - \xlp)) \st x \in \RR^n \}$.
  Given that $H$ supports the epigraph, then $v$ is a subgradient of $g$, in
  particular,
  \[
    g(\xlp) + v^T (x - \xlp) \le g(x) \text{ for every } x \in \RR^n.
  \]
  Let $z(x)$ be the affine function whose graph is $H$, that is, $z(x)
  = g(\xlp) + v^T (x - \xlp)$.
  We now need to show that $g(\xlp) + v^T (x - \xlp) \le 0$ supports $S$ by
  exhibiting an $\xbdr \in S$ such that $g(\xlp) + v^T (\xbdr - \xlp) = 0$.
  By construction,  $z(x_0 + \lambda(\xlp - x_0)) = g(x_0 + \lambda(\xlp
  - x_0))$.
  Since $z(x_0 + \lambda(\xlp - x_0))$ is non-positive for $\lambda = 0$ and
  positive for $\lambda = 1$, it has to be zero for some $\lambda_0$. Let $\xbdr
  = x_0 + \lambda_0(\xlp - x_0)$.
  Therefore, $g(\xbdr) = z(\xbdr) = 0$ and we conclude that $\xbdr \in S$ and
  $g(\xlp) + v^T (\xbdr - \xlp) = 0$.
\end{proof}

Specializing the theorem to differentiable functions directly leads to the
following:
\begin{corollary} \label{cor:supp}
  Let $g \colon \RR^n \to \RR$ be a convex differentiable function, $S = \{
  x \in \RR^n \st g(x) \le 0 \}$, and $\xlp \notin S$.
  Then the valid inequality
  \begin{equation*}
    g(\xlp) + \nabla g(\xlp)^T (x - \xlp) \le 0
  \end{equation*}
  supports $S$, if and only if, there exists $x_0 \in S$ such that $\lambda
  \mapsto g(x_0 + \lambda(\xlp - x_0))$ is affine in $[0,1]$.
\end{corollary}
\begin{proof}
  Since $g$ is differentiable, the subdifferential of $g$ consists only of the
  gradient of $g$.
\end{proof}

A natural candidate for functions with supporting gradient cuts at every point
are functions whose epigraph is a translation of a convex cone.

\begin{example}[Sublinear functions] \label{ex:homo}
  Let $h(x)$ be a sublinear function, that is, convex and positively homogeneous
  function, i.e., $h(\lambda x) = \lambda h(x)$ for any $\lambda \ge 0$.
  For this type of functions, gradient cuts always support $S = \{ x \st
  h(x) \le c \}$, for any $c \geq 0$.
  This follows directly from \Cref{thm:supp}, since $0 \in S$ and $h(\lambda
  \xlp)$ is affine for any $\xlp$.
\end{example}

However, these are not the only functions that satisfy the conditions of
\Cref{thm:supp} for every point.
The previous theorem implies that linearizations always support the constraint
set if a convex constraint $g(x) \leq 0$ is linear in one of its arguments.
\begin{example}[Functions with linear variables] \label{ex:linear_part}
  Let $f \colon \RR^m \times \RR^n \to \RR$ be a convex function of the form
  $f(x, y) = g(x) + a^T y + c$, with $a \neq 0$ and $g \colon \RR^m \to \RR$
  convex.
  Then gradient cuts support $S = \{(x,y) \st f(x,y) \le 0\}$.
  Indeed, assume without loss of generality that $a_1 > 0$ and let $(\xlp, \bar
  y) \notin S$.
  Then there is a $\lambda > 0$ such that $f(\xlp, \bar y - \lambda e_1)
  = g(\xlp) + a^T \bar y + c - a_1 \lambda = 0$.
  The statement follows from \Cref{thm:supp}.

  Consider separating a point $(x_0, z_0)$ from a constraint of the form $z
  = g(x)$ with $g \colon \RR \to \RR$ and convex, with $z_0 < g(x_0)$ (that is,
  separating on the convex constraint $g(x) \leq z$).
  As mentioned earlier, in~\cite{BelottiLeeLibertiMargotWaechter2009} the
  authors suggest projecting $(x_0, z_0)$ to the graph $z = g(x)$ and computing
  a gradient cut there.
  This example shows that this step is unnecessary when the sole purpose is to
  obtain a cut that is supporting to the graph.
\end{example}

In contrast, if $g(x)$ is strictly convex, linearizations at points $x$ such
that $g(x)~\neq~0$ are never supporting to $g(x) \leq 0$.
This follows directly from \Cref{thm:supp} since $\lambda \mapsto g(x + \lambda
v)$ is not affine for any $v$.
We can also characterize convex quadratic functions with supporting
linearizations.
\begin{example}[Convex quadratic functions] \label{ex:quadratic}
  Let $g(x) = x^T A x + b^T x + c$ be a convex quadratic function, i.e., $A$ is
  an $n$ by $n$ symmetric and positive semi-definite matrix.
  We show that gradient cuts support $S = \{ x \in \RR^n \st g(x) \leq 0 \}$, if
  and only if, $b$ is not in the range of $A$, i.e., $b \notin R(A):= \{ Ax \st
  x \in \RR^n \}$.

  First notice that $l_v(\lambda) := g(x + \lambda v)$ is affine linear, if and
  only if, $v \in \ker(A)$.

  Let $v \in \ker(A)$ and $\xlp \notin S$.
  Clearly, there is a $\lambda \in \RR$ such that $\xlp + \lambda v \in S$ if
  and only if $l_v$ is not constant.
  Thus, gradient cuts are \emph{not} supporting, if and only if, $l_v$ is
  constant for every $v \in \ker(A)$.
  But $l_v$ is constant for every $v \in \ker(A)$, if and only if, $b^T v = 0$
  for every $v \in \ker(A)$, which is equivalent to $b \in \ker(A)^{\perp}
  = R(A^T) = R(A)$, since $A$ is symmetric.
  Hence, gradient cuts support $S$, if and only if, $b \notin R(A)$.

  In particular, if $b = 0$, i.e., there are no linear terms in the quadratic
  function, then gradient cuts are never supporting hyperplanes.
  Also, if $A$ is invertible, $b \in R(A)$ and gradient cuts are not supporting.
  This is to be expected since in this case $g$ is strictly convex.
\end{example}

\section{The gauge function} \label{sec:gauge}
Given a \MICP like \eqref{eq:micp_gen}, we can reformulate it to an equivalent
\MICP with a unique constraint for which every linearization supports the
continuous relaxation of the feasible region.
For this, we can use any sublinear function whose 1-sublevel set is $C$.
Each convex set $C$ has at least one sublinear function that represents it,
namely, the \emph{gauge function}~\cite{Rockafellar1970} of $C$.

\begin{definition}
  Let $C \subseteq \RR^n$ be a convex set such that $0 \in \mathring C$.
  The \emph{gauge} of $C$ is
  \[
    \varphi_C (x) = \inf \left\{\  t > 0 \st  x \in t C\ \right\}.
  \]
\end{definition}

The following basic properties of gauge functions make them appealing
for generating supporting hyperplanes.

\begin{prop}[{\cite[Proposition 1.11]{Tuy2016}}]
  \label{prop:gauge_properties}
  Let $C \subseteq \RR^n$ be a convex set such that $0 \in \mathring C$,
  then $\varphi_C(x)$ is sublinear.
  If, in addition, $C$ is closed, then it holds that
  \[
    C = \{ x \in \RR^n \st \varphi_C(x) \le 1 \}
  \]
  and
  \[
    \partial C = \{ x \in \RR^n \st \varphi_C(x) = 1 \}.
  \]
\end{prop}

\Cref{ex:homo} tells us that sublinear functions always generate supporting
hyperplanes.

\subsection{Using the gauge function for separation}

Even though the gauge function is exactly what we need to ensure supporting
gradient cuts, in general, there is no closed-form formula for it.
Therefore, it is not always possible to explicitly reformulate a constraint
$g(x) \leq 0$ as $\varphi(x) \le 1$.

Furthermore, if one is interested in solving mathematical programs with
a numerical solver, performing such a reformulation might introduce some
numerical issues one would have to take care of.
Solvers usually solve up to a given tolerance, that is, they solve $g_j(x) \le
\varepsilon$ for some $\varepsilon > 0$.
Then, even though $C = \{x \st \varphi_C(x) \le 1 \}$, it might be that $\{
x \in \RR^n \st g_j(x) \le \varepsilon \} \nsubseteq \{x \in \RR^n \st
\varphi(x) \le 1 + \varepsilon \}$.
In fact, even simple constraints show this behavior.
Consider $C = \{ x \st x^2 - 1 \leq 0 \}$.
In this case, $\varphi_C(x) = |x|$ and for $x_0 = 1 + \varepsilon$, we have
$\varphi(x_0) = 1 + \varepsilon$.
Then, $x_0$ would be $\varepsilon$-feasible for $\varphi_C(x) \le 1$, although
it would be infeasible for $x^2 -1 \le 0$, since $2 \varepsilon + \varepsilon^2
> \varepsilon$.

Luckily, one does not need to reformulate in order to take advantage of the
gauge function for tighter separation.
The next propositions show how to use the gauge function and a point $\xlp
\notin C$ to obtain a boundary point of $C$ and that linearizing at that
boundary point gives a supporting valid inequality that actually separates
$\xlp$.
For ensuring the existence of a supporting hyperplane we need the following
condition 
\begin{equation} \label{eq:condition}
  \forall j \in J, \forall x \in \partial C, \nabla g_j(x) \neq 0
\end{equation}
For example, this condition is satisfied whenever Slater's condition is
satisfied for \eqref{eq:micp_gen} with $C$ represented by \eqref{eq:std_rep},
that is, when there exists $\xint$ such that $g_j(\xint) < 0$ for every $j \in
J$.

Before we state the propositions we start with a simple lemma.
\begin{lemma} \label{lemma:supp_sepa}
  Let $C \subseteq \RR^n$ be a closed convex set such that $0 \in \mathring C$,
  let $\xbdr \in \partial C$ and $\xlp \notin C$.
  Let $\alpha^T x \leq \beta$ be a valid inequality for $C$ that supports $C$ at
  $\xbdr$.
  If the segment joining $0$ and $\xlp$ contains $\xbdr$, then the inequality
  separates $\xlp$ from $C$.
\end{lemma}
\begin{proof}
  Consider $l(\lambda) = \alpha^T( \lambda \xlp ) - \beta$ and let $\lambda_0
  \in (0,1)$ be such that $\lambda_0 \xlp = \xbdr$.
  The function $l$ is a strictly increasing affine linear function.
  Indeed, $0 \in \mathring C$ implies that $l(0) < 0$, while $l(\lambda_0) = 0$.
  Thus, $l(1) > 0$, i.e., $\alpha^T \xlp > \beta$
\end{proof}

\begin{prop} \label{prop:gauge_boundary}
  Let $C \subseteq \RR^n$ be a closed convex set such that $0 \in \mathring C$
  and let $\xlp \notin C$.
  Then, $\frac{\xlp}{\varphi_C(\xlp)} \in \partial C$.
\end{prop}
\begin{proof}
  First, $\varphi_C(\xlp) \neq 0$ since $\xlp \notin C$.
  The positive homogeneity of $\varphi_C$ implies that
  $\varphi_C\left(\frac{\xlp}{\varphi_C(\xlp)}\right)
  = \frac{\varphi_C(\xlp)}{\varphi_C(\xlp)} = 1$.
  \Cref{prop:gauge_properties} implies $\frac{\xlp}{\varphi_C(\xlp)} \in \partial
  C$.
\end{proof}

Let $J_0(x)$ be the set of indices of the active constraints at $x$, i.e.,
$J_0(x) = \{j \in J \st g_j(x) = 0\}$.

\begin{prop} \label{prop:use_gauge}
  Let $C = \{ x \st g_j(x) \le 0, j \in J\}$ be such that $0 \in \mathring C$
  and let $\varphi_C$ be its gauge function.
  Assume that \eqref{eq:condition} holds.
  Given $\xlp \notin C$, define $\xbdr = \frac{\xlp}{\varphi_C(\xlp)}$.
  Then, for any $j \in J_0(\xbdr)$, the gradient cut of $g_j$ at $\xbdr$ yields
  a valid supporting inequality for $C$ that separates $\xlp$.
\end{prop}
\begin{proof}
  By the previous proposition, we have that $\xbdr \in \partial C$.
  Let $j \in J_0(\xbdr)$.
  Clearly, the gradient cut of $g_j$ at $\xbdr$ yields a valid supporting
  inequality.
  The fact that it separates follows from \Cref{lemma:supp_sepa}.
\end{proof}

Hence, we can get supporting valid inequalities separating a given point $\xlp
\notin C$ by using the gauge function to find the point $\xbdr
= \tfrac{\xlp}{\varphi_C(\xlp)}\in \partial C$.
Then, \Cref{prop:use_gauge} ensures that the gradient cut of any active
constraint at $\xbdr$ will separate $\xlp$ from $C$.
But, how do we compute $\varphi_C(\xlp)$?

\subsection{Evaluating the gauge}
\label{subsec:evaluating}
Let $C = \{ x \st g_j(x) \le 0, j \in J\}$ be a closed convex set such that $0
\in \mathring C$ and consider
\begin{equation} \label{eq:f}
  f(x) = \max_{j \in J} g_j(x).
\end{equation}
In general, evaluating the gauge function of $C$ at $\xlp \notin C$ is
equivalent to solving the following one dimensional equation
\begin{equation} \label{eq:eval_gauge}
  f(\lambda \xlp) = 0,\ \lambda \in (0,1).
\end{equation}
If $\lambda^*$ is the solution, then $\varphi_C(\xlp) = \frac{1}{\lambda^*}$.

One can solve such an equation using a line search.
Note that the line search is looking for a point $\xbdr \in \partial C$ on the
segment between 0 and $\xlp$.
This is exactly what the (extended) supporting hyperplane algorithm performs when it
uses 0 as its interior point.

We would also like to remark that a closed-form formula expression for the gauge
function of $C$ is equivalent to a closed-form formula for the solution of
\eqref{eq:eval_gauge}.
It is possible to find such a formula for some functions, e.g., when $f$ is
a convex quadratic function.

Next, we briefly discuss what happens when 0 is not in the interior of $C$ and
when $C$ has no interior.
In the next section we discuss the implications of the fact that evaluating the
gauge function is equivalent to the line search step of the supporting
hyperplane algorithm.

\subsection{The case $\mathring C = \emptyset$ and using a nonzero interior
point}
When $\mathring C = \emptyset$, we can still use the methods discussed above
using a trick from~\cite{KronqvistLundellWesterlund2016}.
Assuming $C = \{ x \in \RR^n \st g_j(x) \leq 0, j \in J\} \neq \emptyset$,
consider the set $C_{\epsilon} = \{ x \in \RR^n \st g_j(x) \leq \epsilon, j \in
J\}$.
This set satisfies $\mathring C_{\epsilon} \neq \emptyset$ and optimizing
over $C_{\epsilon}$ provides an $\epsilon$-optimal solution.

If $\xint \in \mathring C$ and $\xint \neq 0$, we can translate $C$ so that 0 is in
its interior.
Equivalently, we can build a gauge function centered on $\xint$.
This is given by
\begin{equation*}
  \varphi_{\xint, C}(x) = \varphi_{C-\xint}(x-\xint)
\end{equation*}
Then, given $\xlp \notin C$,
\begin{equation} \label{eq:shift-gauge}
  \xbdr = \frac{\xlp - \xint}{\varphi_{C-\xint}(\xlp-\xint)} + \xint
\end{equation}
belongs to the boundary of $C$.
Equivalently, $\xbdr$ is $\xint + \lambda^* (\xlp - \xint)$, where $\lambda^*$
solves
\begin{equation*}
  f(\xint +\lambda (\xlp - \xint)) = 0,\ \lambda \in (0,1),
\end{equation*}
where $f$ is \eqref{eq:f}.

\section{Convergence proofs} \label{sec:gauge_convergence}
Consider a \MICP given by \eqref{eq:micp_gen} with $C$ represented as
\eqref{eq:std_rep}.
Let $f$ be defined as in \eqref{eq:f}.
As mentioned above, the ESH algorithm~\cite{Veinott1967,
KronqvistLundellWesterlund2016} computes an interior point of $C$ (which we will
assume it to be 0) and performs a line search between $\xlp \notin C$ and 0 to
find a point on the boundary.
It computes a gradient cut at the boundary point, solve the relaxation
again, and repeat the process.
From our previous discussion, computing a gradient cut at the boundary point is
equivalent computing a gradient cut at $\tfrac{\xlp}{\varphi_C(\xlp)}$.
Therefore, the generated cuts are $f(\tfrac{\xlp}{\varphi_C(\xlp)}) + v(x
- \tfrac{\xlp}{\varphi_C(\xlp)}) \leq 0$, where $v \in \partial
f(\tfrac{\xlp}{\varphi_C(\xlp)})$.

To prove the convergence of the ESH algorithm, Veinott~\cite{Veinott1967} and
Kronqvist et al.~\cite{KronqvistLundellWesterlund2016} use tailored arguments.
Here we show that the convergence of the algorithm follows from the convergence
of Kelley's cutting plane algorithm (KCP)~\cite{J.E.Kelley1960}
We note that when $C$ is represented by a convex non-differentiable function,
the KCP algorithm still converges.
One needs to replace gradients by subgradients and one can use any
subgradient~\cite{HorstTuy1990}.
Therefore, given that $\varphi_C(x)$ is a convex function, we know that KCP
converges when applied to $\min \{c^T x \st \varphi_C(x) \leq 1\}$.
Thus, in order to prove that ESH converges, it is sufficient to show that the cutting planes
generated by ESH can also be generated by KCP.

We first prove that the normals of (normalized) supporting valid inequalities
are subgradients of the gauge function at the supporting point.
\begin{lemma} \label{lemma:subgradient_gauge}
  Let $\alpha^T x \leq 1$ be a valid and supporting inequality for $C$.
  Let $\xbdr \in \partial C$ be the point where it supports $C$, i.e., $\alpha^T
  \xbdr = 1$.
  Then $\alpha \in \partial \varphi_C(\xbdr)$
\end{lemma}
\begin{proof}
  We need to show that $\varphi_C(\xbdr) + \alpha^T (x - \xbdr) \leq
  \varphi_C(x)$ for every $x$.
  Note that since $\xbdr \in \partial C$, we have that $\varphi_C(\xbdr) = 1$
  and we just have to prove that $\alpha^T x \leq \varphi_C(x)$

  When $x$ is such that $\varphi_C(x) > 0$, we have $\tfrac{x}{\varphi_C(x)}
  \in C$.
  Due to the validity of $\alpha^T x \leq 1$, it follows that $\alpha^T
  \tfrac{x}{\varphi_C(x)} \leq 1$.

  Now let $x$ be such that $\varphi_C(x) = 0$.
  Then, $\varphi_C(\lambda x) = 0$ for every $\lambda > 0$, i.e., $\lambda x \in
  C$ for every $\lambda > 0$.
  Hence, $\alpha^T (\lambda x) \leq 1$ for every $\lambda > 0$ which implies
  that $\alpha^T x \leq 0 = \varphi_C(x)$.
\end{proof}

Now we prove that the inequalities generated by the ESH algorithm can also be
generated by KCP algorithm, implying the convergence of the ESH algorithm.
\begin{thm}
  Consider a \MICP given by \eqref{eq:micp_gen} with $C$ represented as
  \eqref{eq:std_rep} such that $0 \in \mathring C$ and \eqref{eq:condition}
  holds.
  Let $f$ be defined as in \eqref{eq:f} and let $\xlp \notin C$ be the current
  relaxation solution to separate.
  Let $f(\tfrac{\xlp}{\varphi_C(\xlp)}) + v(x - \tfrac{\xlp}{\varphi_C(\xlp)})
  \leq 0$, with $v \in \partial f(\tfrac{\xlp}{\varphi_C(\xlp)})$, be the
  inequality generated by the ESH algorithm using $0$ as the interior point.
  Then KCP applied to $\min \{c^T x \st \varphi_C(x) \leq 1\}$ can generate the
  same inequality.
\end{thm}
\begin{proof}
  Let us manipulate the inequality obtained by the ESH algorithm.
  First, notice that $f(\tfrac{\xlp}{\varphi_C(\xlp)}) = 0$ and so the
  inequality reads as $v^T x \leq v^T \tfrac{\xlp}{\varphi_C(\xlp)}$.
  Since \eqref{eq:condition} holds, $v \neq 0$.
  Furthermore, by \Cref{lemma:supp_sepa}, $\xlp$ is cut off by the inequality,
  i.e., $v^T \xlp > v^T \tfrac{\xlp}{\varphi_C(\xlp)}$
  This, together with the fact that $\varphi_C(\xlp) > 1$, implies that $v^T \xlp
  > 0$.
  Summarizing, the inequality obtained by the ESH algorithm can be rewritten as
  \[
    \left( \frac{\varphi_C(\xlp)}{v^T \xlp}v \right)^T x \leq 1.
  \]

  \Cref{lemma:subgradient_gauge} implies that $\tfrac{\varphi_C(\xlp)}{v^T
  \xlp}v \in \partial \varphi_C(\tfrac{\xlp}{\varphi_C(\xlp)})$.
  Since $\varphi_C$ is positively homogeneous, $\partial
  \varphi_C(\tfrac{\xlp}{\varphi_C(\xlp)})= \partial \varphi_C(\xlp)$.
  Hence, the same cut can be generated by KCP algorithm applied to $\min \{c^T
  x \st \varphi_C(x) \leq 1\}$ when separating $\xlp$.
\end{proof}

\section{Convex programs represented by non-convex functions}
\label{sec:nonconvex}

In this section we consider problem \eqref{eq:micp_gen} with $C$ represented as
\[
  C = \{ x \st g_j(x) \le 0, j \in J \},
\]
where the functions $g_j$ are differentiable, but not necessarily convex.
As mentioned in the introduction, convex problems represented by non-convex
functions have been considered in~\cite{Dutta2011,Lasserre2009,Lasserre2011}.

The next proposition shows that, under \eqref{eq:condition}, the ESH algorithm
works without modification in this context.
Therefore, its convergence is guaranteed by the convergence of KCP algorithm.
Essentially, we show that with the given representation of $C$ it
is possible to evaluate its gauge function and its subgradients.

Recall that $J_0(x) = \{j \in J \st g_j(x) = 0\}$.
\begin{prop}
  Let $C = \{ x \st g_j(x) \le 0, j \in J\}$ such that $0 \in \mathring C$
  and the function $g_j$ are differentiable.
  Let $\varphi_C$ be the gauge function of $C$.
  For $\xlp \notin C$, define $\xbdr = \frac{\xlp}{\varphi_C(\xlp)}$ and assume
  that \eqref{eq:condition} holds.
  Then, any gradient cut of $g_j$ at $\xbdr$ for any $j \in J_0(\xbdr)$ yields
  a valid supporting inequality for $C$ that separates $\xlp$.
\end{prop}
\begin{proof}
  By \Cref{prop:gauge_boundary} we have that $\xbdr \in \partial C$.
  Let $j \in J_0(\xbdr)$.
  The gradient cut of $g_j$ at $\xbdr$ is $\nabla g_j(\xbdr) (x - \xbdr) \le 0$.

  We first show it is valid, that is, $\forall y \in C, \nabla g_j(\xbdr) (y
  - \xbdr) \le 0$.
  If this is not the case, then there is $y_0 \in C$ for which $\nabla
  g_j(\xbdr) (y_0 - \xbdr) > 0$, i.e., the directional derivative of $g_j$ at
  $\xbdr$ in the direction $y_0 - \xbdr$ is positive.
  Then, there is a small enough $\lambda > 0$ such that $g_j(\xbdr + \lambda
  (y_0 - \xbdr)) > 0$.
  However, the convexity of $C$ implies that $\xbdr + \lambda (y_0 - \xbdr) \in
  C$ for $\lambda \in [0,1]$.
  This contradicts the fact that $g_j(\xbdr + \lambda (y_0 - \xbdr)) > 0$.

  The fact that it separates follows from \Cref{lemma:supp_sepa}.
\end{proof}

This result extends the algorithm of Veinott~\cite{Veinott1967} to further
representations of the set $C$.
The proof of the validity of the cut is the same as the `only if' part
of~\cite[Lemma 2.2]{Lasserre2009}.

\begin{remark}
  Any representation of a convex set $C$ as $\{ x \in \RR^n \st g_j(x) \le 0, j \in
  J\}$ yields a way to evaluate its gauge function, namely,
  \[
    \varphi_C (x) = \inf \left\{ t > 0 \st  \max_j g_j(\frac{x}{t})
    = 0 \right\}.
  \]
  This can be solved using a line search.
  However, what is more important is to be able to compute subgradients.\\
  Given any method to compute subgradients of the gauge function, we can
  apply KCP algorithm using the implicitly defined gauge function.
  This allows us, for example, to drop the requirement that the gradients of the
  active constraints do not vanish at the boundary for solving the problem
  considered in this section.
  This algorithm is more general than the one proposed by
  Lasserre~\cite{Lasserre2011}, but it will not necessarily converge to a KKT
  point of the original problem.
\end{remark}

\section{Concluding remarks} \label{sec:conclusions}
In this paper, we have shown that the extended supporting hyperplane algorithm
studied by Veinott~\cite{Veinott1967} and Kronqvist et
al.~\cite{KronqvistLundellWesterlund2016} is identical to Kelley's classic
cutting plane algorithm applied to a suitable reformulation of the
problem.
We used this new perspective in order to prove the convergence of the method for
the larger class of problems with convex feasible regions represented by
non-convex differentiable constraints.
More generally, the algorithm extends to any representation of a
convex set that allows to compute subgradients of its gauge function.
These theoretical results bear relevance in practice, as the experimental
results in~\cite{KronqvistLundellWesterlund2016} have already demonstrated the
computational benefits of the supporting hyperplane algorithm in comparison to
alternative state-of-the-art solving methods.

\bibliographystyle{abbrv}
\bibliography{gauge}

\end{document}

%% file: 2dim_problem.tex
\begin{tikzpicture}[xscale=1,yscale=1]
    \draw[line width = 0.5pt, blue, fill=blue!15!white] (0,0) circle (1);

    \draw[->] (-2,0) -- (2,0) node[anchor=west] {$x_1$};
    \draw[->] (0,-2) -- (0,2) node[anchor=east] {$x_2$};

    \draw (-1,0) -- (-1,-3pt) node[anchor=north east] {$-1$};
    \draw (1,0) -- (1,-3pt) node[anchor=north west] {$1$};
    \draw (0,-1) -- (-3pt,-1) node[anchor=north east] {$-1$};
    \draw (0,1) -- (-3pt,1) node[anchor=south east] {$1$};

    \node[circle,fill=red,minimum size=3pt,inner sep=0pt] at (1.5, 1.5){};

    \draw[green] plot[domain=-0.1:2.0] (\x ,11.0/6.0 - \x);
\end{tikzpicture}
\hskip 2pt
\begin{tikzpicture}[xscale=1,yscale=1]
    \begin{axis}[
            view={60}{30},
            axis lines = center,
            axis on top,
            xmin=-2, xmax=2, ymin=-2, ymax=2,
            zmin=-4, zmax=1,
            xscale=0.73,yscale=0.73,
            ztick={\empty},
        ]

        \addplot3[domain=0:2*pi,samples y=0,mark=none,blue,fill=blue!15!white,thick]
            ({cos(deg(x))},{sin(deg(x))},{0.});

        \addplot3[surf,opacity=0.2,z buffer=sort,domain=-4:1,
            y domain =0:2*pi,colormap/blackwhite]
            (
            {sqrt(1-x)*cos(deg(y))},
            {sqrt(1-x)*sin(deg(y))},
            x);

        \addplot3[mark options={color=red}, mark=*]
            coordinates {(1.5,1.5,0)};

        \addplot3[mark=none, dashed, color=yellow]
            coordinates {(1.5,1.5,0) (1.5,1.5,-3.5)};

        \addplot3[mark options={color=gray}, mark=*]
            coordinates {(1.5,1.5,-3.5)};

        \addplot3[surf, opacity=0.5,green,domain=0:2.0, y domain = 0:2] (x , y, 11.0/2.0 - 3*x - 3*y);

        \addplot3[line width = 1pt,green,domain=-0.1:2.0, samples y = 0] (x ,11.0/6.0 - x, 0);
    \end{axis}
\end{tikzpicture}

%% file: gauge.bbl
\begin{thebibliography}{10}

\bibitem{BelottiLeeLibertiMargotWaechter2009}
P.~Belotti, J.~Lee, L.~Liberti, F.~Margot, and A.~W{\"a}chter.
\newblock Branching and bounds tightening techniques for non-convex {MINLP}.
\newblock {\em Optimization Methods \& Software}, 24(4-5):597--634, 2009.

\bibitem{DuranGrossmann1986}
M.~A. Duran and I.~E. Grossmann.
\newblock An outer-approximation algorithm for a class of mixed-integer
  nonlinear programs.
\newblock {\em Mathematical Programming}, 36(3):307--339, oct 1986.

\bibitem{Dutta2011}
J.~Dutta and C.~S. Lalitha.
\newblock Optimality conditions in convex optimization revisited.
\newblock {\em Optimization Letters}, 7(2):221--229, Oct. 2011.

\bibitem{EronenMaekelaeWesterlund2012}
V.-P. Eronen, M.~M. Mäkelä, and T.~Westerlund.
\newblock On the generalization of {ECP} and {OA} methods to nonsmooth convex
  {MINLP} problems.
\newblock {\em Optimization}, 63(7):1057--1073, aug 2012.

\bibitem{FletcherLeyffer1994}
R.~Fletcher and S.~Leyffer.
\newblock {Solving mixed integer nonlinear programs by outer approximation}.
\newblock {\em Mathematical Programming}, 66(1):327--349, 1994.

\bibitem{Gomory1958}
R.~E. Gomory.
\newblock Outline of an algorithm for integer solutions to linear programs.
\newblock {\em Bulletin of the American Mathematical Society}, 64(5):275--279,
  sep 1958.

\bibitem{HorstTuy1990}
R.~Horst and H.~Tuy.
\newblock {\em Global Optimization}.
\newblock Springer Nature, 1990.

\bibitem{J.E.Kelley1960}
J.~J.~E.~Kelley.
\newblock The cutting-plane method for solving convex programs.
\newblock {\em Journal of the Society for Industrial and Applied Mathematics},
  8(4):703--712, dec 1960.

\bibitem{KronqvistLundellWesterlund2016}
J.~Kronqvist, A.~Lundell, and T.~Westerlund.
\newblock The extended supporting hyperplane algorithm for convex mixed-integer
  nonlinear programming.
\newblock {\em Journal of Global Optimization}, 64(2):249--272, 2016.

\bibitem{Lasserre2009}
J.~B. Lasserre.
\newblock On representations of the feasible set in convex optimization.
\newblock {\em Optimization Letters}, 4(1):1--5, oct 2009.

\bibitem{Lasserre2011}
J.~B. Lasserre.
\newblock On convex optimization without convex representation.
\newblock {\em Optimization Letters}, 5(4):549--556, apr 2011.

\bibitem{LubinBienstockVielma2015}
M.~Lubin, D.~Bienstock, and J.~P. Vielma.
\newblock Two-sided linear chance constraints and extensions.
\newblock {\em arXiv preprint arXiv:1507.01995}, 2015.

\bibitem{Rockafellar1970}
R.~T. Rockafellar.
\newblock {\em Convex analysis}.
\newblock Princeton university press, 1970.

\bibitem{Tuy2016}
H.~Tuy.
\newblock {\em Convex Analysis and Global Optimization}.
\newblock Springer International Publishing, 2016.

\bibitem{Veinott1967}
A.~F. Veinott.
\newblock The supporting hyperplane method for unimodal programming.
\newblock {\em Operations Research}, 15(1):147--152, feb 1967.

\bibitem{WesterlundPettersson1995}
T.~Westerlund and F.~Pettersson.
\newblock An extended cutting plane method for solving convex {MINLP} problems.
\newblock {\em Computers {\&} Chemical Engineering}, 19:131--136, jun 1995.

\end{thebibliography}
